\documentclass[a4paper]{amsart}
\usepackage{etex}
\usepackage{enumerate}
\usepackage[english]{babel}

\usepackage{graphicx}
\usepackage{multicol}
\usepackage[bottom]{footmisc}
\usepackage[export]{adjustbox}
\usepackage{functan}
\usepackage[latin1]{inputenc}
\usepackage{cancel}
\usepackage{calligra}
\usepackage{colortbl}
\usepackage{multirow}
\usepackage{enumerate}
\usepackage{varioref}
\usepackage{booktabs}
\usepackage{amscd}
\usepackage{color}
\usepackage{colortbl}
\usepackage{pstricks}
\usepackage{pst-all}
\usepackage{mparhack}
\usepackage{amssymb}
\usepackage{amsmath}
\usepackage{dsfont}
\usepackage{mathrsfs}
\usepackage{amsfonts}
\usepackage{amssymb}
\usepackage[mathcal]{euscript}
\usepackage[all,cmtip]{xy}
\usepackage{amssymb}
\usepackage{amsthm}
\theoremstyle{definition}
\newtheorem{definition}{Definition}[section]
\newtheorem{example}[definition]{Example}
\newtheorem{remark}[definition]{Remark}
\newtheorem{notation}[definition]{Notation}

\theoremstyle{plain}
\newtheorem*{ThmA}{Theorem A}

\newtheorem{theorem}[definition]{Theorem}
\newtheorem{proposition}[definition]{Proposition}
\newtheorem{lemma}[definition]{Lemma}

\numberwithin{equation}{section}

\def \alt96 {`}

\def \R {\mathbb{R}}

\def \loc {\mathrm{loc}}

\def \N {\mathds{N}}
\def \c {\mathbf{c}}

\def \G {\mathbb{G}}

\def \dela {{\delta_\lambda}}

\def \c {\mathbf{c}}
\def \LL {{\mathcal{L}}}

\def \d {\mathrm{d}}
\def \de {\partial}

\def \Lie {\mathrm{Lie}}

\def \LL {\mathcal{L}}

\begin{document}
  \author{Stefano Biagi}
 \address{Stefano Biagi: Dipartimento di Ingegneria Industriale e Scienze Matematiche,
 Università Politecnica delle Marche,  via Brecce Bianche 12, I-60131 Ancona, Italy.}
 \email{s.biagi@dipmat.univpm.it}
  \author{Andrea Bonfiglioli}
 \address{Andrea Bonfiglioli: Dipartimento di Matematica,
 Alma Mater Studiorum - Università di Bologna,
 Piazza Porta San Donato 5, I-40126 Bologna, Italy.}
 \email{andrea.bonfiglioli6@unibo.it}
  \author{Marco Bramanti}
 \address{Marco Bramanti: Dipartimento di Matematica, Politecnico di Milano,
 Via Bonardi 9, I-20133 Milano, Italy.}
 \email{marco.bramanti@polimi.it}

 \title[Global Sobolev Estimates for Homogeneous PDO's]{Global estimates in Sobolev spaces\\ for homogeneous H\"ormander sums of squares}
\begin{abstract}
 Let $\LL=\sum_{j=1}^m X_j^2$ be a H\"ormander sum of squares of vector fields in space $\R^n$,
 where any $X_j$ is homogeneous of degree $1$ with respect to a family
 of non-isotropic dilations in space.
 In this paper we prove
 global estimates and regularity properties for $\LL$ in
 the $X$-Sobolev spaces $W^{k,p}_X(\R^n)$, where
 $X = \{X_1,\ldots,X_m\}$. In our approach,
 we combine local results for general H\"ormander sums of squares,
 the homogeneity property of the $X_j$'s, plus
 a global lifting technique for homogeneous vector fields.
\end{abstract}
\maketitle



\section{Introduction and statement of the result}

Let $X_{1},\ldots,X_{m}$ be a set of smooth and linearly in\-de\-pen\-dent\footnote{The linear independence of the $X_i$'s
 is meant with respect to the vector space of the smooth vector fields on $\R^n$;
 this must not be confused with the linear independence of the
 vectors $X_1(x),\ldots,X_m(x)$ in $\R^n$
 (when $x\in\R^n$):
 the latter is sufficient but not necessary to the former linear independence.
 Thus, $X_1=\de_{x_1}$ and $X_2=x_1\,\de_{x_2}$ are linearly independent vector fields, even if $X_1(0,x_2)\equiv (1,0)$ and $X_2(0,x_2)\equiv(0,0)$
 are dependent vectors of $\R^2$.}
vector fields on $\mathbb{R}^{n}$, satisfying the following assumptions:\medskip

\begin{itemize}
\item[(H.1)] there exists a family of (non-isotropic) dilations
 $\{\delta_{\lambda}\}_{\lambda>0}$ of the form
$$
 \delta_{\lambda}:\mathbb{R}^{n}\longrightarrow\mathbb{R}^{n}\qquad
 \delta_{\lambda}(x)=(\lambda^{\sigma_{1}}x_{1},\ldots,\lambda^{\sigma_{n}}x_{n}),
$$
 where $1=\sigma_{1}\leq\cdots\leq\sigma_{n}$ are integers such that
 the $X_{i}$'s are $\delta_{\lambda}$-homogeneous of degree $1$:
$$
 X_{j}(f\circ\delta_{\lambda})=\lambda\,(X_{j}f)\circ\delta_{\lambda},\quad\forall\,\,\lambda>0,\,\,f\in C^{\infty}(\mathbb{R}^{n}),\quad j=1,\ldots,m;
$$
  In what follows, we denote by $q:=\sum_{j=1}^{m}\sigma_{j}$ the so-called homogeneous dimension
  of $(\R^n,\dela)$.

\item[(H.2)] $X_{1},\ldots,X_{m}$ satisfy H\"{o}rmander's rank condition at $0$, i.e.,
$$
 \dim\left\{  Y(0):Y\in\mathrm{Lie}(X)\right\}  =n,
$$
 where $\mathrm{Lie}(X)$ is the smallest Lie sub-algebra of the Lie algebra of the smooth vector fields on $\R^n$ 
 which contains $X:=\{X_{1},\ldots,X_{m}\}$.\medskip
\end{itemize}
 Some remarks on our assumptions are in order. Assumption (H.1) implies that, if
 $$ X_{j}=\sum_{k=1}^{n} b_{j,k}(x)\,\partial_{x_{k}},$$
 then $b_{j,k}(x)$ must be a polynomial function, $\delta_{\lambda}$-homogeneous of degree $\sigma_{k}-1$.
 Incidentally, this straightforwardly implies that
\begin{equation}\label{notazioneXIprecede}
 b_{j,k}(x)=b_{j,k}(x_1,\ldots,x_{k-1})\quad \text{for any $j\leq m$ and $k\leq n$,}
\end{equation}
 or, more precisely, $b_{j,k}(x)$ depends on those $x_i$'s such that $\sigma_i\leq \sigma_{k}-1$.
 From \eqref{notazioneXIprecede} we infer that the formal adjoint of $X_j$ is $-X_j$.
 Let us fix some notation. For any multi-index $I=(i_{1},\ldots ,i_{k})$ with $i_1,\ldots,i_k\in\left\{1,2,\ldots,m\right\}$, we let
\begin{equation}\label{notazioneXI}
 X_{I}=X_{i_{1}}X_{i_{2}}\cdots X_{i_{k}},\quad
 {X}_{[I]}=\left[\left[{X}_{i_{1}},{X}_{i_{2}}\right],\ldots,{X}_{i_{k}}\right],\qquad |I|=k.
\end{equation}
 When $k=1$ and $I=(i_1)$, we agree to let $X_I=X_{i_1}$. It is easy to check that, by (H.1),
 the operators $X_I$ and $X_{[I]}$ are $\dela$-homogeneous of degree $|I|$.  The $\dela$-homogeneity
 of the vector field $X_{[I]}$ is equivalent to the identity
\begin{equation}\label{equidellahomogvf}
 X_{[I]}(\dela(x))=\lambda^{-|I|}\dela(X_{[I]}(x)),\quad \forall\,\,\lambda>0,\,\,x\in\R^n.
\end{equation}
\begin{remark}[Global H\"ormander condition]\label{hormandereverywhere}
   We observe that, by (H.1) and (H.2), the validity of H\"{o}rmander's rank condition at $0$ implies its validity at any other point
 $x\in\mathbb{R}^{n}$.
 Indeed, the iterated (left nested) brackets $X_{[I]}$ span $\mathrm{Lie}(X)$. Hence, by (H.2), we can find a family $X_{[I_1]},\ldots,X_{[I_n]}$
 such that  $X_{[I_1]}(0),\ldots,X_{[I_n]}(0)$ is a basis of $\R^n$.
 Thus, the matrix-valued function
  $$z\mapsto \mathbf{M}(z) := \big(X_{[I_1]}(z)\cdots X_{[I_n]}(z)\big)$$
 is non-singular at $z=0$; therefore, there exists a neighborhood $\Omega$ of $0$ such that
   $\det(\mathbf{M}(z)) \neq 0$ for every $z \in \Omega$.
 Fixing $x \in \R^n$ and taking a small $0<\lambda \ll 1$ such that $\delta_\lambda(x) \in \Omega$, we have
   \begin{align*}
   0\neq  \det \big(\mathbf{M}\big(\delta_\lambda(x)\big)\big)&\stackrel{\eqref{equidellahomogvf}}{=}
    \det\Big(\lambda^{-|I_1|}\,\delta_\lambda\big(X_{[I_1]}(x)\big)\cdots
    \lambda^{-|I_n|}\,\delta_\lambda\big(X_{[I_n]}(x)\big)\Big) \\
    &\,\,\,=\,\, \lambda^{-|I_1|-\cdots-|I_n|}\,
    \det\Big(\delta_\lambda\big(X_{[I_1]}(x)\big)\cdots
    \delta_\lambda\big(X_{[I_n]}(x)\big)\Big).
   \end{align*}
 This implies that the vectors
    $\delta_\lambda\big(X_{[I_1]}(x)\big),\ldots,
    \delta_\lambda\big(X_{[I_n]}(x)\big)$ form a basis of $\R^n$, so that
    the same is true of $X_{[I_1]}(x),\ldots,X_{[I_n]}(x)$, since the linear map
    $\delta_\lambda$ is an isomorphism of $\R^n$.
 This proves that $X_1,\ldots,X_m$ satisfy H\"ormander's rank condition at any $x\in\R^n$.
\end{remark}
 Thus, by H\"{o}rmander's Theorem \cite{Hor2}, the \emph{homogeneous sums of squares}
 $$ \LL =\sum_{j=1}^{m}X_{j}^{2} $$
 is $C^\infty$-hypoelliptic on every open set $\Omega\subseteq\mathbb{R}^{n}$, which means that
 every distributional solution $u$ of an equation $Lu=f$ in $\Omega$ is smooth on
 every sub-domain $\Omega'\subseteq\Omega$ where $f$ is smooth. From \eqref{notazioneXIprecede}
 we also infer that $L$ is formally self-adjoint. Note that the case $q=2$ implies that
 $\LL$ is a strictly elliptic constant-coefficient operator on $\R^2$, so that it is not restrictive to assume that $q>2$.
\begin{example}\label{exa.grushin1}
 In $\mathbb{R}^{2}$, let us consider
 $$ X_{1}  =\partial_{x_{1}},\quad X_{2}  =x_{1}\,\partial_{x_{2}};\qquad \delta_{\lambda}(x_{1},x_{2})   =(\lambda x_{1},\lambda^{2}x_{2}).$$
 Condition (H.1) is easily checked. Here $n=2$, $q=3$ and
 $$ \LL=X_{1}^{2}+X_{2}^{2}=\partial_{1,1}+x_{1}^{2}\,\partial_{2,2}.$$
 Condition (H.2) holds because $X_1$ and $[X_{1},X_{2}]=\de_{x_2}$ give a basis of $\R^2$ at any point.
\end{example}
\begin{example}\label{exa.grushin2}
 More generally, in $\mathbb{R}^{2}$, let
 $$ X_{1} =\partial_{x_{1}},\quad X_{2}   =x_{1}^{k}\,
 \partial_{x_{2}};\qquad \delta_{\lambda}(x_{1},x_{2})    =(\lambda x_{1},\lambda^{k+1}x_{2}).$$
 Again, (H.1) is easy to check. Here $n=2$, $q=k+2$  and
 $$
 \LL=X_{1}^{2}+X_{2}^{2}=\partial_{1,1}+x_{1}^{2k}\,\partial_{2,2}.$$
 Condition (H.2) holds true as well because $X_1$ and
 $$\frac{1}{k!}[X_{1},[X_{1},\ldots [X_{1},X_{2}]]]=\de_{x_2}\qquad \text{(bracket of length $k+1$)}$$
 span $\R^2$ at any point.
\end{example}
\begin{example}
 In $\mathbb{R}^{n}$, let us consider
 $$X_{1} =\partial_{x_{1}},\quad X_{2} =x_{1}\partial_{x_{2}}+x_{2}\partial_{x_{3}}+\cdots+x_{n-1}\partial_{x_{n}};\qquad
  \delta_{\lambda}(x)  =(\lambda x_{1},\lambda^{2}x_{2},\cdots,\lambda^{n}x_{n}).$$
 (H.1) is easily checked. Note that $q=n(n+1)/2>n$ and
 $$ \LL=X_{1}^{2}+X_{2}^{2}=\partial_{1,1}+( x_{1}\partial_{x_{2}}+x_{2}\partial_{x_{3}}+\ldots +x_{n-1}\partial_{x_{n}})^{2}.$$
Condition (H.2) holds because
\begin{align*}
\partial_{x_{1}}  &  =X_{1}\\
\partial_{x_{2}}  &  =\left[  X_{1},X_{2}\right] \\
&  \vdots\\
\partial_{x_{n}}  &  =\left[  \left[ \left[  X_{1},X_{2}\right]  ,X_{2}\right],  \ldots,X_{2}\right] \quad \text{(bracket of length $n$).}
\end{align*}
\end{example}
 Let $\Omega\subseteq\R^n$ be an open set. Following the notation in \eqref{notazioneXI}, the Sobolev spaces with respect to the system of vector fields $X$ are defined,
  for $p\in (1,\infty)$ and $k\in \mathbb{N}\cup\{0\}$, by setting
 $$W_{X}^{k,p}(\Omega):=
  \Big\{u\in L^{p}(\Omega):\,X_{I}u\in L^{p}(\Omega),\,\,\text{for any $I$  with $|I|\leq k$}\Big\},$$
 endowed with the norm
 $$
 \Vert u\Vert_{W_{X}^{k,p}(\Omega)}:=
  \Vert u\Vert_{L^{p}(\Omega)}+\sum_{\vert I\vert \leq k}\Vert X_{I}u\Vert_{L^{p}(\Omega)}.
$$
 Here the derivatives $X_{I}u$ exist, \emph{a priori}, in the weak sense at least.
 When $k=0$, it is understood that $X_Iu=u$ for any multi-index with $|I|\leq 0$, so that
  $(W_{X}^{k,p}(\Omega ), \Vert \cdot\Vert_{W_{X}^{k,p}(\Omega )})$
  is just the usual normed space $(L^p(\Omega), \Vert \cdot\Vert_{L^p(\Omega )})$.\medskip

 We are interested in establishing global regularity results in the scale of
 these Sobolev spaces for homogeneous sums of squares $\LL$. Namely, our main result is
 the following:

\begin{theorem}[Global regularity for homogeneous sums of squares]\label{Ch8-Thm Sobolev Caso A}
 Let $\LL$ be as above, under assumptions
 \emph{(H.1)-(H.2)} on the vector fields $X_1,\ldots,X_m$.

 Let also $p\in( 1,\infty) $ and let $k$ be a nonnegative integer.
 Then, there exists $\Lambda = \Lambda_{k,p}>0$
 such that, if $u\in L^{p}(\mathbb{R}^{n}) $ and $Lu\in W_{X}^{k,p}(\mathbb{R}^{n})$
  \emph{(}which means that the distribution $\LL u$ can
  be identified with a function in $W_{X}^{k,p}(\mathbb{R}^{n})$\emph{)},
  then $u\in W_{X}^{k+2,p}(\mathbb{R}^{n})$ and
\begin{equation}\label{Ch8-Global estimate}
 \|u\|_{W_{X}^{k+2,p}(\mathbb{R}^{n})}\leq
 \Lambda_{k,p}\,\Big(\|\LL u\|_{W_{X}^{k,p}(\mathbb{R}^{n})}
 + \|u\|_{L^{p}(\mathbb{R}^{n})}\Big).
\end{equation}
\end{theorem}
 \noindent This theorem will be proved in section 3, throughout Theorems
 \ref{Thm global estimates} and \ref{Thm regularity}.\medskip

 Theorem \ref{Ch8-Thm Sobolev Caso A} is well known if the sum of squares $\LL$ is not just $\dela$-homogeneous of degree $2$,
 but also left invariant with respect to a Lie group operation; more precisely,
 if $\LL$ is a \emph{sub-Laplacian on a Carnot group}: in this case the above
 result is due to Folland, see \cite[Thm. 6.1]{Fo2}. Let us review the
 definition of this key concept, since it will play an important role in the following:
\begin{definition}
 We say that $\mathbb{G}=(\mathbb{R}^{N},\ast,D_{\lambda}) $ is a
 (homogeneous) Carnot group if:
 \begin{enumerate}
   \item $\ast$ is a Lie group operation in $\mathbb{R}^{N}$ (that we qualify as \textquotedblleft translations\textquotedblright)
   and, for some fixed positive integer exponents $\alpha_{1},\ldots ,\alpha_{N}$, the maps
  $$ D_{\lambda}( x) =( \lambda^{\alpha_{1}}x_{1},\ldots ,
  \lambda^{\alpha_{N}}x_{N})\quad \text{for $\lambda>0$}$$
 form a family of group automorphisms (that we qualify as \textquotedblleft dilations\textquotedblright).\medskip

   \item Let $X_{i}$ (for $i=1,2,\ldots ,N$) be the only left invariant vector field
   which agrees with $\partial_{x_{i}}$ at the origin; moreover, let $H$ be the set of the vector fields among
   $X_{1},\ldots ,X_{N}$ which are $D_{\lambda}$-homogeneous of degree $1$; then the set $H$ satisfies
  H\"{o}rmander's condition at the origin (hence, by left-invariance, at every point of $\mathbb{R}^{N}$).
 \end{enumerate}
 In this case, if $H=\{Z_1,\ldots,Z_m\}$, the sub-Laplacian operator on $\G$ defined by
 $\Delta_\G=\sum_{j=1}^{m} Z_{j}^{2}$
 is $D_\lambda$-homogeneous of degree $2$, left invariant, and $C^\infty$-hypoelliptic.
\end{definition}
 For a technical reason that will become apparent in a moment (see \eqref{campisiliftanoomogenicosa}),
 we do not require that the exponents $\alpha_k$'s of the dilations $D_\lambda$ be increasingly ordered
 (as is done e.g., in \cite{BLUlibro}).\medskip

 In the more general case of the so-called \textquotedblleft sums of squares
 of H\"{o}rmander's vector fields\textquotedblright, defined on some domain
 $\Omega\subseteq\mathbb{R}^{n}$ but not necessarily homogeneous with respect
 to any family of dilations, nor necessarily left invariant with respect to any
 Lie-group translations, a regularity result such as Theorem \ref{Ch8-Thm Sobolev Caso A} is known only in
 a local form. Namely, Rothschild-Stein  proved the following:
\begin{ThmA}[Interior regularity for H\"ormander sum of squares, \protect{\cite[Thm. 16]{RS}}]
 Let
 $X_{1},\ldots ,X_{m}$ be a system of smooth vector fields satisfying
 H\"{o}rmander's condition in some domain $\Omega\subseteq\mathbb{R}^{n}$, and
 let $L=\sum_{i=1}^{m}X_{i}^{2}$. Finally, let $k$ be a nonnegative integer and $p\in( 1,\infty)$.\medskip

 Then the following facts hold:
 \begin{itemize}
  \item[{(i)}] if $u$ is
  any distribution in $\Omega$ with $Lu\in
 W_{X}^{k,p}( \Omega)$, then $u\in W_{X,\mathrm{loc}}^{k+2,p}(\Omega)$;
 \item[{(ii)}] for any domains $\Omega'\Subset\Omega''\Subset\Omega$, it is possible to find
 a constant $c_{k,p} > 0$ such that
 \begin{equation}\label{local a priori}
 \|u\|_{W_{X}^{k+2,p}(\Omega')}\leq
 c_{k,p}\Big\{\| L u\|_{W_{X}^{k,p}(\Omega'')}+
  \|u\|_{L^{p}(\Omega'')}\Big\},
\end{equation}
for every distribution $u$ in $\Omega$ with $Lu\in W^{k,p}_X(\Omega)$.
 \end{itemize}
\end{ThmA}
 Incidentally, we note that for general H\"{o}rmander operators
 $\sum_{i=1}^{m}X_{i}^{2}+X_{0}$ with drift term $X_0$ (with $X_{0},X_{1},\ldots ,X_{m}$
 satisfying H\"{o}rmander's condition in $\Omega$),
 only the basic estimate (\ref{local a priori}) for $k=0$ is known, while a
 complete regularity theory in the scale of Sobolev spaces $W_{X}^{k,p}$ is so far lacking.\footnote{Rothschild-Stein
 \cite{RS} state the result, but with no proof, and the methods in \cite{RS} do not seem to adapt easily to the drift case.
 We have not been able to locate any proof of Theorem A for $\sum_{i=1}^{m}X_{i}^{2}+X_{0}$ in the existing literature.}\medskip

 Coming back to the case of the sums of squares  $L=\sum_{i=1}^m X_i^2$,
 if the vector fields $X_{1},\ldots ,X_{m}$
 satisfy H\"{o}rmander's condition in $\Omega=\mathbb{R}^{n}$, it is quite
 natural to ask whether the result of Theorem A can
 be improved to that of Theorem \ref{Ch8-Thm Sobolev Caso A} without assuming
 the Carnot group structure. However, only a few results in this direction seem
 to be known, so far. Bramanti, Cupini, Lanconelli, Priola in \cite{BCLP} have
 studied a class of Ornstein-Uhlenbeck operators of the kind
 $$
 L u=\sum_{i,j=1}^{m}a_{i,j}u_{x_{i},x_{j}}+\sum_{i,j=1}^{N}b_{i,j}x_{i}u_{x_{j}}\quad \text{in $\R^N$},
 $$
 with $m<N$, $(a_{i,j}) _{i,j=1}^{m}$ a constant, symmetric,
 positive-definite matrix, and $(b_{i,j}) _{i,j=1}^{N}$ a constant
 matrix satisfying a suitable structure assumption. This operator can be
 rewritten in the form of a H\"{o}rmander operator
 $L u=\sum_{i=1}^{m}X_{i}^{2}u+X_{0}$ on the whole of $\mathbb{R}^{N}$; however, this $L$ is neither left invariant nor
 (in general) homogeneous with respect to any family of dilations. For these
 operators the following global estimates are proved (just in the basic case $k=0$)
$$
 \sum_{i=1}^{m}\Vert u_{x_{i},x_{j}}\Vert _{L^{p}(\mathbb{R}^{N}) }+
  \Vert X_{0}u\Vert _{L^{p}(\mathbb{R}^{N}) }\leq c\left\{  \Vert L u\Vert _{L^{p}(\mathbb{R}^{N}) }+
   \Vert u\Vert _{L^{p}( \mathbb{R}^{N}) }\right\},\quad \text{for $1<p<\infty$.}$$
   Apart from this result, and its extension to continuous
 variable coefficients $a_{i,j}$ contained in \cite{BCLP2}, no global Sobolev
 estimates for classes of H\"{o}rmander operators which do not fulfill
 Folland's assumptions of both left-invariance and homogeneity seem to be known.

 Therefore the present result Theorem \ref{Ch8-Thm Sobolev Caso A} seems to be interesting in its own right,
 although its proof is not difficult. The simple idea is to apply Rothschild-Stein's local Sobolev estimates,
 and then to exploit the dilations to get global ones. In doing this, however, one also requires some global
 interpolation inequalities for Sobolev norms, which are so far available in the case of Carnot groups only.
 Establishing these inequalities in the present context is possible in view of some deep
 result dealing with a global lifting of homogeneous vector fields to a higher dimensional Carnot group.
 This lifting result is a powerful tool, first developed by Folland \cite{Fo3} and, in the form that we actually need, by
 two of us, \cite{BiBo}.
 We start (in Section \ref{sec:liftingandintepolation}) by reviewing this lifting procedure, then we establish
 suitable interpolation inequalities, and finally (in Section \ref{sec:mainproof}) we prove our main result.

\section{Lifting and interpolation inequalities}\label{sec:liftingandintepolation}

 The following result is proved in \cite{BiBo}, by using Folland's lifting in \cite{Fo3} plus a convenient
 change of variable turning the lifting into an explicit projection.
\begin{theorem}[Global Lifting]\label{Thm lifting}
 Assume that $X=\{X_{1},\ldots,X_{m}\}$ satisfy
 \emph{(H.1)} and \emph{(H.2)}. Let $N:=\dim(\Lie\{X\})$. We denote the points of $\mathbb{R}^{N}
 \equiv\mathbb{R}^{n}\times\mathbb{R}^{s}$ by $(x,\xi)$ \emph{(}if $N=n$, we agree that the $\xi$ variable does not appear\emph{)}. Then, the following facts hold:
\begin{enumerate}
  \item[\emph{(1)}]  There exist a Carnot group $\mathbb{G}=(\mathbb{R}^{N},\ast,D_{\lambda})$
 and a system $\{\widetilde{X}_{1},\ldots,\widetilde{X}_{m}\}$ of
 Lie-generators of $\mathrm{Lie}(\mathbb{G})$ such that $\widetilde{X}_{i}$ is
 a lifting of $X_{i}$ for every $i=1,\ldots,m$, that is:
\begin{equation}\label{campisiliftano}
 \widetilde{X}_{i}(x,\xi)=X_{i}(x)+R_{i}(x,\xi),
\end{equation}
 where $R_{i}(x,\xi)$ is a smooth vector field operating only in the variable
 $\xi\in\mathbb{R}^{s}$, with coefficients possibly depending on $(x,\xi)$.\medskip

  \item[\emph{(2)}]
  The dilations $\left\{  D_{\lambda}\right\}  _{\lambda>0}$ \emph{(}which
 make the $\widetilde{X}_{i}$'s homogeneous of degree $1$\emph{)} and the dilations
 $\left\{\delta_{\lambda}\right\}_{\lambda>0}$ \emph{(}which make the $X_{i}$'s homogeneous
 of degree $1$\emph{)} are related as follows:
\begin{equation}\label{campisiliftanoomogenicosa}
 D_{\lambda}( x,\xi) =( \delta_{\lambda}( x),\delta^*_{\lambda}( \xi) ),
\end{equation}
 with $\delta^*_{\lambda}( \xi) =( \lambda^{\tau_{1}}\xi_{1},\ldots ,\lambda^{\tau_{s}}\xi_{s})$,
 for suitable integers $\tau_{s}\geq\cdots \geq\tau_{1}\geq 1$.
\end{enumerate}
\end{theorem}
\begin{remark}[The case $N=n$]\label{rem.sulladimenzoinen}
 Since $X$ is a H\"ormander system in $\R^n$, one has $N\geq n$.
 As a matter of fact, Theorem \ref{Thm lifting} has been  proved in \cite{BiBo} under the assumption $N>n$. By a recent result in \cite{BBCCM}, Theorem \ref{Thm lifting}
 also holds in the case $N=n$. Indeed, if the latter holds, we have that:
\begin{itemize}
  \item $\Lie\{X\}$ is an $n$-dimensional Lie algebra of analytic vector fields in $\R^n$ (analyticity follows from the fact that the $X_j$'s have
  polynomial component functions, due to (H.1));

  \item $X$ is a H\"ormander system, due to (H.2) (see also Remark \ref{hormandereverywhere});

  \item any vector field $Y\in \Lie\{X\}$ is complete, i.e., the integral curves of $Y$ are defined on the whole of $\R$
  (this can be easily proved as a consequence of (H.1) and \eqref{notazioneXIprecede}).
\end{itemize}
 Under these three conditions, a result in \cite{BBCCM} proves that $\Lie\{X\}$ coincides with the Lie algebra
 of a Lie group $\G$ on $\R^n$. As a matter of fact, under assumption (H.1), this Lie group $\G$ turns out to be a homogeneous Carnot group
 with dilations $\dela$ (see e.g., \cite[Chapter 16]{BiagiBonfBook}), so that Theorem \ref{Thm lifting} holds without the need to perform any further lifting.
\end{remark}
\begin{remark}[Rothschild-Stein's lifting vs.\,\,\,Folland's lifting]
 The first famous result
 about the lifting of vector fields was proved by Rothschild-Stein in
 \cite{RS}. They showed that every system of H\"{o}rmander's vector fields can
 be lifted, locally, to a higher dimensional system of free H\"{o}rmander's
 vector fields, which can be locally approximated, in a suitable sense, by the
 generators of a Carnot group. In the above Theorem \ref{Thm lifting}, instead, the initial
 system is directly lifted to the generators of a Carnot group $\G$, the process
 being performed globally, while $\G$ needs not be a free group. These advantages are made possible by the
 homogeneity of the original vector fields.
\end{remark}
\begin{example}\label{sec.th:ex_Grushin}
   Let us consider the vector fields $X_1,X_2$ in Example \ref{exa.grushin1}.
   The associated Carnot group according to Theorem \ref{Thm lifting} is $\G = (\R^3,\ast,D_\lambda)$ with
\begin{equation*}
    D_\lambda(x_1,x_2,\xi_1) = (\lambda x_1,\lambda^2 x_2, \lambda\,\xi_1),
\end{equation*}
  while the composition law is
  \begin{equation*}
   (x_1,x_2,\xi_1)\ast(x'_1,x'_2,\xi'_1) = (x_1+x'_1, x_2+x'_2+x_1\xi'_1, \xi_1+\xi'_1).
  \end{equation*}
  Furthermore, the vector fields $\widetilde{X}_1,\widetilde{X}_2$ lifting $X_1$ and $X_2$ are
  \begin{equation} \label{sec.th:exGrushinZs}
   \widetilde{X}_1 = \de_{x_1}, \qquad \widetilde{X}_2 = x_1\,\de_{x_2}+\de_{\xi_1}.
  \end{equation}
  The operator $\LL = X_1^2+X_2^2$ lifts to
  the sub-Laplacian $\Delta_\G = \widetilde{X}_1^2+\widetilde{X}_2^2$.
  The latter is (modulo a change of variable)
  the Kohn-Laplacian on the first Heisenberg group.
 \end{example}
 \begin{example}\label{sec.th:exGrushinII}
   Let us consider the vector fields $X_1,X_2$ in Example \ref{exa.grushin2},  in the case when $k=2$.
   The associated Carnot group according to Theorem \ref{Thm lifting} is $\G = (\R^4,\ast,D_\lambda)$ with
 \begin{equation*}
    d_\lambda(x_1,x_2,\xi_1,\xi_2)
    = (\lambda x_1,\lambda^3 x_2, \lambda\,\xi_1,\lambda^2\,\xi_2),
  \end{equation*}
 and the composition law $(x_1,x_2,\xi_1,\xi_2)\ast(x'_1,x'_2,\xi'_1,\xi'_2)$ is
 \begin{equation*}
  \begin{split}
    \Big(x_1+x'_1, x_2 + x'_2 + x_1(x_1+x'_1)\xi'_1+2x_1\xi'_2,
   \xi_1+\xi'_1, \xi_2+\xi'_2+\tfrac{1}{2}(x_1\xi'_1-x'_1\xi_1) \Big).
   \end{split}
  \end{equation*}
  The vector fields $\widetilde{X}_1,\widetilde{X}_2$ lifting $X_1$ and $X_2$ are
 \begin{equation*}
   \widetilde{X}_1 = \de_{x_1}-\frac{\xi_1}{2}\,\de_{\xi_2},
   \qquad \widetilde{X}_2 = x_1^2\,\de_{x_2}+\de_{\xi_1}+\frac{x_1}{2}\,\de_{\xi_2}.
\end{equation*}
 \end{example}
 Following the notation in Theorem \ref{Thm lifting}, in the lifted space $\mathbb{R}^{N}$ we can consider the Sobolev spaces
 $W_{\widetilde{X}}^{k,p}$, where $\widetilde{X}=\{\widetilde{X}_1,\ldots,\widetilde{X}_m\}$. On the other hand, when $\widetilde{X}_{i}$ acts on
 a function $f$ only depending on the variables $x$, one simply gets
$$
\widetilde{X}_{i}f( x) =X_{i}f( x),\quad i=1,\ldots,m .
$$
 This suggests that these Sobolev spaces simply project onto the spaces
 $W_{X}^{k,p}$. However, when computing $L^{p}$ norms, some care must be taken
 about the domain of the functions involved.
 In Proposition \ref{Prop projected norms}
 we shall compare $L^{p}$
 norms in suitable balls of the original space and in the lifted variables. Let
 us first fix some notation and basic facts.

 The dilations $\dela$ in $\mathbb{R}^{n}$ induce a \emph{homogeneous norm $\Vert\cdot\Vert$} in
 $\mathbb{R}^{n}$ as follows: by definition, we let $\Vert 0\Vert =0$, and, for every $x\in\mathbb{R}^{n}\setminus\{0\}$, we
 define $\Vert x\Vert $ as the unique positive number such as
 $$\Big|\delta_{1/\Vert x\Vert }( x) \Big|=1,$$
 where $\vert \cdot\vert $ stands for the Euclidean norm. This
 definition makes sense since, for every $x\neq0$, the function
 $(0,\infty)\ni\lambda\mapsto\vert \delta_{\lambda}( x) \vert$
 is continuous, strictly increasing, and its image set is $( 0,\infty) $.
\begin{remark}\label{charatceriz.Vert}
 Let $\mathbb{S}^{n-1}\subset \R^n$ denote, as usual, the unit sphere $\{x\in\R^n:\,|x|=1\}$. Then $\Vert\cdot\Vert$ is characterized by any of the following
 equivalent conditions:
 \begin{enumerate}
   \item for any $\lambda>0$, the level set $\{x\in\R^n:\,\|x\|=\lambda\}$ coincides with $\dela(\mathbb{S}^{n-1})$ (the latter being the
   ellipsoid with semi-axes $\lambda^{\sigma_1},\ldots,\lambda^{\sigma_n}$) which is the set described by the equation
   \begin{equation}\label{eq.ellips}
    \frac{x_1^2}{\lambda^{2\sigma_1}}+\frac{x_2^2}{\lambda^{2\sigma_2}}+\cdots+\frac{x_n^2}{\lambda^{2\sigma_n}}=1;
   \end{equation}

   \item $\Vert \cdot \Vert$ coincides with the unique map $u:\R^n\to [0,\infty)$ which is $\dela$-homogeneous of degree $1$ and such that
   $$u(x) =1\quad \text{if and only if}\quad |x|=1; $$

   \item for any $x\neq 0$, $\|x\|$ is the reciprocal of the unique positive solution $t$ to the algebraic equation
   $$x_1^2\,t^{2\sigma_1}+\cdots+x_n^2\,t^{2\sigma_n}=1; $$

   \item for any $x\neq 0$, $\|x\|$ is the reciprocal of the unique $\lambda>0$ for which the $\dela$-line through $x$, that is the set
   $\{\dela(x):\lambda>0\}$, intersects the sphere $\mathbb{S}^{n-1}$.
  \end{enumerate}
\end{remark}
 \noindent Thus $\Vert \cdot \Vert$ enjoys the following properties:
\begin{align*}
 \Vert x\Vert  &  \geq0 \quad \text{and}\quad ( \Vert x\Vert
 =0\Leftrightarrow x=0),\\
 \Vert \delta_{\lambda}(x) \Vert  &  =\lambda\,
 \Vert x\Vert\quad \text{for every $\lambda>0$ and every $x\in\R^n$.}
\end{align*}
Also, since the exponents $\sigma_{i}$ appearing in the dilations are positive
integers, the function $x\mapsto\Vert x\Vert $ is smooth outside
the origin. (This can be seen by applying the Implicit Function Theorem to the
function $f( \lambda,x) =\vert \delta_{\lambda}(x) \vert ^{2}-1$).

 Analogously we can define in $\mathbb{R}^{N}$ and in $\R^s$ two homogeneous norms by means of the
 dilations $\left\{  D_{\lambda}\right\}  _{\lambda>0}$ and $\left\{  \delta^*_{\lambda}\right\}  _{\lambda>0}$
 introduced in Theorem \ref{Thm lifting}, and these homogeneous norms enjoy
 similar properties of the ones established for the pair $(\R^n,\{\dela\}_{\lambda>0})$.
 By a small abuse of notation we shall denote with the same
 symbol $\Vert \cdot\Vert $ these three homogeneous norms defined in
 $\mathbb{R}^{n}$, $\mathbb{R}^{N}$ and $\mathbb{R}^s$. They are related by the
 following facts (which holds by point 2 in Theorem \ref{Thm lifting}):
\begin{equation}\label{proprietanormette}
 \Vert ( x,\xi) \Vert    \geq\Vert (x,0) \Vert =\Vert x\Vert ;\qquad \Vert ( x,\xi) \Vert   \geq\Vert (0,\xi) \Vert =\|\xi\|.
\end{equation}
 We will define the following balls centered at the origins of $\R^n$, $\R^N$ and $\R^s$ respectively:
\begin{align*}
 B_{r}( 0)  &=\left\{  x\in\mathbb{R}^{n}:\Vert x\Vert <r\right\},\\
 \widetilde{B}_{r}( 0)  &=\left\{  (x,\xi)\in\mathbb{R}^{N}:\Vert( x,\xi) \Vert <r\right\},\\
 B^*_{r}( 0)  &=\left\{  \xi\in\mathbb{R}^{s}:\Vert \xi\Vert <r\right\},
\end{align*}
 and we note that, due to \eqref{proprietanormette},
 $B_{r}( 0)$ is the projection of $\widetilde{B}_{r}(0)$ via the canonical projection of $\R^N=\R^n\times \R^s$ onto $\R^n$.
 It is not difficult to prove that
\begin{align}
    {B}_{r}(0)&=\bigg\{x\in\R^n:\,\,\frac{x_1^2}
    {r^{2\sigma_1}}+\cdots+\frac{x_n^2}{r^{2\sigma_n}}<1\bigg\}, \label{relazionecartespalle2}\\
    \widetilde{B}_{r}(0)&=\bigg\{(x,\xi)\in\R^n\times\R^s:\,\,
    \frac{x_1^2}{r^{2\sigma_1}}+\cdots+\frac{x_n^2}{r^{2\sigma_n}}
    +\frac{\xi_1^2}{r^{2\tau_1}}+\cdots+\frac{\xi_s^2}{r^{2\tau_s}}<1\bigg\}, \label{relazionecartespalle3}\\
     {B}^*_{r}(0)&=\bigg\{\xi\in\R^s:\,\,\frac{\xi_1^2}
    {r^{2\tau_1}}+\cdots+\frac{\xi_s^2}{r^{2\tau_s}}<1\bigg\}, \label{relazionecartespalle4}
\end{align}
  which means that ${B}_{r}(0)$,  $\widetilde{B}_{r}(0)$ and $B^*_r(0)$ are the bounded open sets whose boundaries are the ellipsoids
  with equations analogous to \eqref{eq.ellips} (relative to the dilations $\delta_r$, $D_\lambda$ and $\delta^*_r$ respectively). Equivalently, if $D,\widetilde{D},D^*$
  denote (respectively) the open Euclidean balls with center at the origin and radius $1$ in $\R^n,\R^N,\R^s$ (respectively), then, for any $r>0$ one has
  $$B_r(0)=\delta_r(D),\qquad \widetilde{B}_r(0)=D_r(\widetilde{D}),\qquad {B}^*_r(0)=\delta^*_r(D^*). $$
  Starting from \eqref{relazionecartespalle2}-to-\eqref{relazionecartespalle4} one can prove that (for any $r>0$)
\begin{align}
 \widetilde{B}_r(0)&\subseteq {B}_r(0)\times {B}^*_r(0),\label{relazionecartespalle5}\\
  \widetilde{B}_r(0)&\supseteq {B}_{r/2}(0)\times {B}^*_{r/2}(0).\label{relazionecartespalle6}
\end{align}
 Indeed, \eqref{relazionecartespalle5} is a consequence of
 $$\max\bigg\{\frac{x_1^2}{r^{2\sigma_1}}+\cdots+\frac{x_n^2}{r^{2\sigma_n}},
    \frac{\xi_1^2}{r^{2\tau_1}}+\cdots+\frac{\xi_s^2}{r^{2\tau_s}}\bigg\}\leq
 \frac{x_1^2}{r^{2\sigma_1}}+\cdots+\frac{x_n^2}{r^{2\sigma_n}}
    +\frac{\xi_1^2}{r^{2\tau_1}}+\cdots+\frac{\xi_s^2}{r^{2\tau_s}}<1, $$
 whereas \eqref{relazionecartespalle6} is a consequence of
 $$1\geq \sum_{j=1}^n \frac{x_j^2}{(r/2)^{2\sigma_j}}=
  \sum_{j=1}^n \frac{2^{2\sigma_j}x_j^2}{r^{2\sigma_j}}\geq 2^{2\sigma_1} \sum_{j=1}^n \frac{x_j^2}{r^{2\sigma_j}}>2\sum_{j=1}^n \frac{x_j^2}{r^{2\sigma_j}}, $$
 together with an analogous inequality involving $\xi$'s and $\tau$'s; here we also used
 $$1\leq \sigma_1\leq \cdots\leq \sigma_n,\quad 1\leq \tau_1\leq \cdots\leq \tau_s.$$
 Throughout the paper, we shall occasionally use the simplified notation $B_r$ for any set $B_r(0)$.
\begin{example}\label{normette.grushin}
 Consider the vector fields $X_1,X_2$ in Example \ref{sec.th:ex_Grushin}.
 The dilations in $\R^2$ and in the lifted space $\R^3$ are respectively
 $$\dela(x_1,x_2)=(\lambda x_1,\lambda^2x_2),\qquad D_\lambda(x_1,x_2,\xi_1)=(\lambda x_1,\lambda^2x_2,\lambda\xi_1).$$
 Thus, by using for example the characterization (3) in Remark \ref{charatceriz.Vert}, one can obtain the explicit expressions for
 the homogeneous norms in the un-lifted and lifted spaces:
 \begin{align*}
    \Vert(x_1,x_2)\Vert &=
           \frac{1}{\sqrt2}\,{\sqrt{\sqrt{x_1^4+4\,x_2^2}+x_1^2}}\,,\\
               \Vert(x_1,x_2,\xi_1)\Vert &=
           \frac{1}{\sqrt2}\,{\sqrt{\sqrt{(x_1^2+\xi_1^2)^2+4\,x_2^2}+x_1^2+\xi_1^2}}\,.
  \end{align*}
\end{example}
 We have the following result, concerning $L^p$-norms in
 $B_{r}( 0) $ and $\widetilde{B}_{r}( 0) $:
\begin{lemma}\label{Prop projected norms}
 With the above notation, for any function
 $u( x) $ of $n$ variables defined in $B_{r}( 0) $,
 let us define the corresponding function $\widetilde{u}$ of $N$ variables by setting
\begin{equation}\label{Prop projected norms.EQ0}
\widetilde{u}( x,\xi) =u( x),\quad (x,\xi)\in B_r(0)\times \R^s.
\end{equation}
 Then, for every $p\in[ 1,\infty)$ and $r>0$, we have
\begin{equation}\label{Prop projected norms.EQ1}
 c_{1}\Vert u\Vert _{L^{p}( B_{r/2}( 0) )}\leq\Vert \widetilde{u}\Vert _{L^{p}( \widetilde{B}_{r}( 0) ) }\leq c_{2}\Vert u\Vert_{L^{p}( B_{r}( 0) ) },
\end{equation}
 where \emph{(}denoting by $\mathrm{meas}$ the Lebesgue measure in $\R^s$\emph{)}
 $$c_1=c_1(r,p)=\mathrm{meas}(B^*_{r/2}(0))^{1/p},\qquad
 c_2=c_2(r,p)=\mathrm{meas}(B^*_{r}(0))^{1/p}.$$
\end{lemma}
 Note that \eqref{Prop projected norms.EQ1} makes sense, since $\widetilde{B}_r(0)\subset B_r(0)\times \R^s$, due to
 \eqref{proprietanormette}. From Lemma \ref{Prop projected norms} and \eqref{campisiliftano}, we immediately infer that
 (if $\widetilde{u}$ is as in \eqref{Prop projected norms.EQ0})
\begin{equation}\label{Prop projected norms.EQ1postilla}
 u\in W_{X}^{k,p}(B_r(0))\quad \Longleftrightarrow\quad \widetilde{u}\in W_{\widetilde{X}}^{k,p}(\widetilde{B}_r(0)).
\end{equation}
 Indeed, from \eqref{Prop projected norms.EQ0} we get that $\widetilde{X}_I \widetilde{u}=\widetilde{X_I u}$
 on $\widetilde{B}_r(0)$, for any multi-index $I$.
\begin{proof}
 We have the following computation, based on \eqref{relazionecartespalle5}:
\begin{align*}
\Vert \widetilde{u}\Vert _{L^{p}( \widetilde{B}_{r}(0) ) }^{p}  &  =\iint_{\widetilde{B}_{r}( 0) }\vert
 u( x) \vert ^{p}\,\d x\,\d\xi\leq
 \iint_{{B}_{r}( 0)\times B^*_r(0) }\vert
 u( x) \vert ^{p}\,\d x\,\d\xi=c_{2}( r) \int_{B_{r}( 0) }\vert u(x) \vert ^{p}\,\d x,
\end{align*}
 where $c_2(r)$ is the Lebesgue measure in $\R^s$ of $B^*_r(0)$.
 On the other hand, by \eqref{relazionecartespalle6},
\begin{align*}
 \Vert \widetilde{u}\Vert _{L^{p}( \widetilde{B}_{r}(0) ) }^{p}  &  \geq\iint_{{B}_{r/2}(0)\times {B}^*_{r/2}(0)}
  \vert u( x) \vert ^{p}\,\d x\,\d\xi = c_{1}( r) \int_{B_{r/2}( 0) }\vert
 u( x) \vert ^{p}\,\d x,
\end{align*}
 where $c_1(r)$ is the Lebesgue measure in $\R^s$ of $B^*_{r/2}(0)$. This completes the proof.
\end{proof}
With the above result at hand, we can now prove the following useful:
\begin{proposition}[Global interpolation inequality]\label{Prop global interpolation}
 For every $p\in( 1,\infty)$ there exists
 $\c_p > 0$ such that, for every $u\in W_{X}^{2,p}(\mathbb{R}^{n})$ and
 every $\varepsilon > 0$, one has
\begin{equation}\label{Prop global interpolationequata}
 \Vert X_{i}u\Vert _{L^{p}(\mathbb{R}^{n})  }
  \leq\varepsilon\,\Vert X_{i}^{2}u\Vert _{L^{p}(\mathbb{R}^{n})}+
   \frac{\c_p}{\varepsilon}\,\Vert u\Vert _{L^{p}(\mathbb{R}^{n})}\quad
    \text{for $i=1,2,\ldots ,m$.}
\end{equation}
\end{proposition}
\begin{proof}
 For simplicity, we write $B_r,\widetilde{B}_r$ instead of  $B_r(0),\widetilde{B}_r(0)$.

 If, as usual, $\widetilde{X}_{i}$ is the lifted vector field of $X_i$ in the Carnot group $\mathbb{G}$, by
 known interpolation inequalities in Carnot groups
 (see \cite[Thm. 21]{BBto}), we know that
  (for some constant $\widetilde{c}_p > 0$)
$$
 \Vert \widetilde{X}_{i}v\Vert _{L^{p}( \widetilde{B}_{1/2}) }\leq
 \sigma\,\Vert \widetilde{X}_{i}^{2}v\Vert_{L^{p}( \widetilde{B}_{1}) }+
 \frac{\widetilde{c}_p}{\sigma}\Vert v\Vert _{L^{p}( \widetilde{B}_{1}) }\quad
 \text{for every $v\in W_{\widetilde{X}}^{2,p}( \widetilde{B}_{1}) $ and every $\sigma\in(0,1)$.}$$
 Let us apply this inequality to a function $v=\widetilde{w}$, where $w$
 depends only on $x$: by Lemma \ref{Prop projected norms} (see also
 \eqref{Prop projected norms.EQ1postilla}), for every $w\in W_{X}^{2,p}( B_{1}) $ and any $\sigma\in (0,1)$
  we get
$$
 \Vert X_{i}w\Vert _{L^{p}( B_{1/4}) }\leq
 c'_p\bigg(
 \sigma\,\Vert X_{i}^{2}w\Vert _{L^{p}( B_{1})}+\frac{\widetilde{c}_p}{\sigma}\,
 \Vert w\Vert _{L^{p}( B_{1}) }\bigg),
 \quad \text{where $c'_p := \frac{\mathrm{meas}(B^*_{1/4}(0))^{1/p}}
 {\mathrm{meas}(B^*_1(0))^{1/p}}$.}
$$
 Next,
 let us apply the last inequality to $w(x) :=u( \delta_{R}(x) ) $,
 where $u\in W_{X}^{2,p}( B_{R} ) $. We find:
$$
 R^{1-q/p}\Vert X_{i}u\Vert _{L^{p}( B_{R/4}) }\leq
 c'_p\,R^{2-q/p}\,\sigma\Vert X_{i}^{2}u\Vert _{L^{p}( B_{R})}+
 \frac{c_p''\,R^{-q/p}}{\sigma}\Vert u\Vert _{L^{p}( B_{R})},
$$
 for every $u\in W_{X}^{2,p}( B_{R}( 0) ) $
 (and $c''_p := c'_p\,\widetilde{c}_p$).
 After dividing by $R^{1-q/p}$, this gives
\begin{equation}\label{Prop global interpolation.EQinter1}
  \Vert X_{i}u\Vert _{L^{p}( B_{R/4}) }\leq
 c'_p\,R\,\sigma\Vert X_{i}^{2}u\Vert _{L^{p}( B_{R})}+\frac{c''_p}{R\,\sigma}\Vert u\Vert _{L^{p}( B_{R})},\quad
 \text{for every $u\in W_{X}^{2,p}( B_{R} ) $.}
\end{equation}
 For every fixed $\varepsilon > 0$ and every $R>2\,\varepsilon/c'_p$,
 let us take $\sigma=\varepsilon/(c'_pR)<1/2$ in \eqref{Prop global interpolation.EQinter1}: we obtain
\begin{equation}\label{interp furba}
 \Vert X_{i}u\Vert _{L^{p}( B_{R/4}) }\leq  \varepsilon\Vert X_{i}^{2}u\Vert _{L^{p}( B_{R}) }
 +\frac{\c_p}{\varepsilon}\Vert u\Vert _{L^{p}( B_{R}) } \qquad
 (\text{with $\c_p = c'_pc''_p$}),
\end{equation}
 for every $u\in W_{X}^{2,p}( B_{R})$.
 Hence, given $u\in W_{X}^{2,p}(\mathbb{R}^{n})$, letting $R\rightarrow\infty$ in \eqref{interp furba}
 (and noticing that $\c_p$ is independent of
 $u$ and $R$), we get
 at once \eqref{Prop global interpolationequata}.
\end{proof}
\begin{notation}\label{derivatose}
 Henceforth, we shall use the following compact notation (where $i\geq 1$ is integer):
\begin{align*}
 \Vert Du\Vert _{L^{p}(\Omega)}    =\sum_{j=1}^{m}\Vert X_{j}u\Vert _{L^{p}( \Omega) },\qquad
 \Vert D^{i}u\Vert _{L^{p}( \Omega) }  &  =\sum_{\vert I\vert =i}\Vert X_{I}u\Vert _{L^{p}( \Omega) }.
\end{align*}
 We also let $D^0u=u$. Notice that $\Vert u\Vert_{W^{k,p}_X(\Omega)}=\sum_{i=0}^ k \Vert D^iu\Vert_{L^p(\Omega)}$.
\end{notation}
 In the sequel, we shall also need the following local version of the interpolation inequality:
\begin{proposition}\label{Prop local interpolation}
 For fixed $p\in( 1,\infty)$, $R>0$ and $u\in W_{X}^{2,p}( B_{R}( 0) )$, let
 \begin{equation} \label{eq.defiPhik}
 \Phi_{k}(u):=\sup_{\sigma\in(0,1)} \Big\{\big((1-\sigma)R\big)^k\,
 \Vert D^{k}u\Vert _{L^{p}( B_{\sigma R}( 0))}\Big\},\qquad \text{for $k=0,1,2$.}
 \end{equation}
 There exists $\alpha_{p}>0$
 independent of $u$ and $R$ such that, for every $\varepsilon \in(0,1]$, one has
 \begin{equation} \label{eq.localinterpolationPhi}
 \Phi_{1}(u)\leq\varepsilon\,\Phi_{2}(u)+\frac{\alpha_{p}}{\varepsilon}\,\Phi_{0}(u).
 \end{equation}
\end{proposition}
 In order to prove Proposition \ref{Prop local interpolation}, we need the following:
\begin{lemma}[Radial cutoff functions]\label{Lemma cutoff}
 For every $r_1,r_2\in(0,\infty)$, with $r_1 < r_2$,
 there exists a cut-off function
 $\phi\in C_{0}^{\infty}(\R^n)$, valued in $[0,1]$, with the following properties:
  \begin{itemize}
   \item[{(i)}] $\phi\equiv 1$ on $B_{r_1}(0)$;
   \item[{(ii)}] $\phi \equiv 0$ outside $B_{r_2}(0)$;
   \item[{(iii)}] for any $j\in\mathbb{N}$ there exists a constant $\varrho_j > 0$, independent
  of $r_1$ and $r_2$, such that
  \begin{equation} \label{eq.estimDkphi}
   \|D^j\phi\|_{L^\infty(\R^n)}\leq \frac{\varrho_j}{(r_2-r_1)^j}.
   \end{equation}
  \end{itemize}
\end{lemma}
\begin{proof}
 We leave it to the reader to check that the following choice of $\phi$ does the job:
 $$\phi(x)=\chi\bigg(\frac{\|x\|}{2\,(r_2-r_1)}-\frac{r_1+r_2}{4(r_2-r_1)}\bigg), $$
 where $\chi:\R\to \R$ is a $C^\infty$-function with the following properties:
 $\phi$ is decreasing, $\phi\equiv 1$ on $(-\infty,-1/4]$, $\phi\equiv 0$ on $[1/4,\infty)$.
 (The smoothness of $\phi$ is a consequence of the fact that $\|\cdot\|$ is smooth outside the origin.)
 \end{proof}
  It is worthwhile noting that, in the present context, we are able to build
 cut-off functions adapted to any ball \emph{centered at the origin} (but not at any point).
 \begin{proof}[Proof of Proposition \ref{Prop local interpolation}.]
  We arbitrarily take
  $\sigma\in (0,1)$ and we let $\phi\in C_0^\infty(\R^n)$ be a cut-off fun\-ction
  as in Lemma \ref{Lemma cutoff}, with
  $r_1 := \sigma R$ and
  $
  r_2 := \sigma'R$ (where $\sigma' := (1+\sigma)/2<1$).

  Since, by assumption, $u\in W_X^{2,p}(B_R)$, it
  is straightforward to check that $v := \phi\, u\in W_X^{2,p}(\R^n)$
  (note that $v\equiv u$ on $B_{\sigma R}$).
  Thus,
  if $\delta$ is \emph{any} positive real number, from
  Proposition \ref{Prop global interpolation} we obtain
  \begin{equation} \label{eq.tocombinewithsuccessiva}
  \begin{split}
   \|Du\|_{L^p(B_{\sigma R})}
   & \,\,\,=\,\,\,\sum_{i = 1}^m\|X_i u\|_{L^p(B_{\sigma R})}=\sum_{i = 1}^m\|X_i v\|_{L^p(B_{\sigma R})}
   \leq \sum_{i = 1}^m\|X_i v\|_{L^p(\R^n)} \\
   & \stackrel{\eqref{Prop global interpolationequata}}{\leq}
   \delta\,\sum_{i = i}^m\|X_i^2 v\|_{L^p(\R^n)}
   + \frac{m\,\c_p}{\delta}\,\|v\|_{L^p(\R^n)},
  \end{split}
  \end{equation}
  where $\c_p > 0$ is a suitable constant independent of $u,\delta$ and $\sigma$.
  We then observe that, by taking into account the properties
  of $\phi$ in Lemma \ref{Lemma cutoff}, one has
  \begin{equation} \label{eq.estimvLp}
   \|v\|_{L^p(\R^n)} =   \|\phi\, u\|_{L^p(\R^n)}
   \leq \|u\|_{L^p(B_{\sigma'R})};
  \end{equation}
  moreover, for every index $i\in\{1,\ldots,m\}$, we also have
  \begin{equation} \label{eq.tocombinewithprima}
  \begin{split}
   & \|X_i^2 v\|_{L^p(\R^n)} = \|X_i^2(\phi\, u)\|_{L^p(\R^n)}
   = \bigg\Vert u\,X_i^2\phi+2\,(X_i\phi)\,(X_iu)+ \phi\,X_i^2u \bigg\Vert_{L^p(B_{\sigma'R})} \\
   & \quad \leq \frac{4\varrho_2}{\big((1-\sigma)R\big)^2}\,\|u\|_{L^p(B_{\sigma'R})}
   + \frac{4\varrho_1}{(1-\sigma)R}\,\|X_iu\|_{L^p(B_{\sigma'R})}
   + \varrho_0\|X_i^2u\|_{L^p(B_{\sigma'R})};
   \end{split}
  \end{equation}
  here, $\varrho_0,\varrho_1, \varrho_2$ are
  the constants appearing in
  \eqref{eq.estimDkphi}, which are independent of $\sigma$ and $R$. Multiplying
  both sides of \eqref{eq.tocombinewithsuccessiva}
  by $(1-\sigma)R > 0$, and using estimates
  \eqref{eq.estimvLp}-\eqref{eq.tocombinewithprima}, we get
  \begin{equation*}
   \begin{split}
   (1-\sigma)R\,\|Du\|_{L^p(B_{\sigma R})}
    & \leq \varrho_0\,\delta \,(1-\sigma)R\,\|D^2 u\|_{L^p(B_{\sigma'R})}
   +4\varrho_1\,\delta \,\|D u\|_{L^p(B_{\sigma'R})} \\
   & \quad + m\,(\c_p+4\varrho_2)\,
   \bigg\{\frac{\delta}{(1-\sigma)R}+\frac{(1-\sigma)R}{\delta}\bigg\}\,\|u\|_{L^p(B_{\sigma'R})}.
   \end{split}
  \end{equation*}
   Setting $\theta_p:=\varrho_0+4\,\varrho_1+m\,(\c_p+4\varrho_2)$, this gives
 \begin{equation} \label{eq.dovesceglieredelta}
   \begin{split}
   (1-\sigma)R\,\|Du\|_{L^p(B_{\sigma R})}
    & \leq \theta_p\,\delta \,(1-\sigma)R\,\|D^2 u\|_{L^p(B_{\sigma'R})}
   +\theta_p\,\delta \,\|D u\|_{L^p(B_{\sigma'R})} \\
   & \quad + \theta_p\,
   \bigg\{\frac{\delta}{(1-\sigma)R}+\frac{(1-\sigma)R}{\delta}\bigg\}\,\|u\|_{L^p(B_{\sigma'R})},
   \end{split}
  \end{equation}
  Now, if $\varepsilon\in(0,1]$ is arbitrarily fixed,
  since \eqref{eq.dovesceglieredelta} holds for every
  $\delta > 0$, we can choose in particular
  $$\delta = \delta_\varepsilon := \frac{(1-\sigma)R\,\varepsilon}{8\,\theta_p} > 0.$$
  Thanks to this choice of $\delta$, \eqref{eq.dovesceglieredelta} becomes
  \begin{equation} \label{eq.estimtotakesup}
   \begin{split}
    & (1-\sigma)R\,\|Du\|_{L^p(B_{\sigma R})}
     \leq \\
     & \frac{\varepsilon}{8}\,\big((1-\sigma)R\big)^2\,
    \|D^2 u\|_{L^p(B_{\sigma'R})}
    + \frac{\varepsilon}{8}\,(1-\sigma)R\,\|D u\|_{L^p(B_{\sigma'R})}
    + \bigg(\frac{\varepsilon}{8}+\frac{8\,\theta_p^2}{\varepsilon}\bigg)\,\|u\|_{L^p(B_{\sigma'R})}.
   \end{split}
  \end{equation}
  Bearing in mind that
  \begin{equation} \label{eq.sigmaprimedausare}
   \sigma'=\frac{1+\sigma}{2}\in(0,1),\qquad\text{so that}\qquad
  (1-\sigma)R=2(1-\sigma')R ,
  \end{equation}
   the above \eqref{eq.estimtotakesup} can be rewritten as
  \begin{equation*}
   \begin{split}
    & (1-\sigma)R\,\|Du\|_{L^p(B_{\sigma R})}
     \leq \\
     & \frac{\varepsilon}{2}\,\big((1-\sigma')R\big)^2\,
    \|D^2 u\|_{L^p(B_{\sigma'R})}
    + \frac{\varepsilon}{4}\,(1-\sigma')R\,\|D u\|_{L^p(B_{\sigma'R})}
    + \bigg(\frac{\varepsilon}{8}+\frac{8\,\theta_p^2}{\varepsilon}\bigg)\,\|u\|_{L^p(B_{\sigma'R})}.
   \end{split}
  \end{equation*}
  Taking the supremum over $\sigma\in(0,1)$ on both sides of the latter inequality, one gets
 \begin{align*}
  \Phi_1(u) \leq \frac{\varepsilon}{2}\,\Phi_2(u)+ \frac{\varepsilon}{4}\Phi_1(u)
  + \bigg(\frac{\varepsilon}{8}+\frac{8\,\theta_p^2}{\varepsilon}\bigg)\Phi_0(u).
 \end{align*}
 As a consequence, since $\varepsilon\in(0,1]$, we obtain
 \begin{align*}
  \frac{3}{4}\,\Phi_1(u) & \leq \bigg(1-\frac{\varepsilon}{4}\bigg)\Phi_1(u)
  \leq \frac{\varepsilon}{2}\,\Phi_2(u)+\bigg(\frac{\varepsilon}{8}+\frac{8\,\theta_p^2}{\varepsilon}\bigg)\Phi_0(u),
 \end{align*}
 from which we derive that
 \begin{align*}
  \Phi_1(u) & \leq \frac{2\,\varepsilon}{3}\,\Phi_2(u) + \frac{3}{4}\bigg(\frac{\varepsilon}{8}+
  \frac{8\,\theta_p^2}
  {\varepsilon}\bigg)\Phi_0(u)
  \leq \varepsilon\,\Phi_2(u)+\frac{\alpha_p}{\varepsilon}\,\Phi_0(u),
 \end{align*}
 where $\alpha_p := \dfrac{3(64\,\theta_p^2+1)}{32}$, which is the desired \eqref{eq.localinterpolationPhi}.
 \end{proof}
\section{Global estimates and regularity results}\label{sec:mainproof}
 In this last section we provide the proof of our main result,
 Theorem \ref{Ch8-Thm Sobolev Caso A}.
 To begin with, we establish the following lemma, of independent interest.
 \begin{lemma} \label{lem.Dksobolev}
  Let $p\in(1,\infty)$ and let $k$ be a nonnegative integer.
  There exists a positive constant $\Theta_{k,p} > 0$,
  only depending on $k$ and $p$, such that
  \begin{equation} \label{eq.DkSobolevLemma}
   \|D^{i+2}u\|_{L^p(\R^n)} \leq \Theta_{k,p}\,\|D^i(\LL u)\|_{L^p(\R^n)}\qquad \text{for $i\in\{0,\ldots,k\}$,}
  \end{equation}
  for every function $u\in W^{k+2,p}_X(\R^n)$.
  As usual,  $\LL u = \sum_{j = 1}^mX_j^2u$.
 \end{lemma}
 \begin{proof}
  Let $i\in\{0,\ldots,k\}$ be fixed. For every $R > 0$, we consider
  the function $v_R := u\circ\delta_R$. Since, by assumption,
  $u$ belongs to $W^{k+2,p}_X(\R^n)$ (and $\delta_R$ is linear),
  it is easy to see that
  $$v_R\in W^{k+2,p}_X(\R^n)\subseteq W^{i+2,p}_X(\R^n).$$
  Thus, since $\LL = \sum_{j = 1}^mX_j^2$ is a H\"ormander
  sum of squares in $\R^n$, we are entitled to apply Theorem A
  for $v_R\in W^{i+2,p}_X(\R^n)$, with  $\Omega' := B_1(0)$ and $\Omega'' := B_2(0)$, obtaining
  (for some $c_{i,p} > 0$)
  \begin{equation} \label{eq.estimThmAvR}
   \|D^{i+2} v_R\|_{L^p(B_1)}
  \leq \|v_R\|_{W^{i+2,p}_X(B_1)}
  \leq c_{i,p}\,\Big\{\|\LL v_R\|_{W_{X}^{i,p}(B_2)}+
  \|v_R\|_{L^{p}(B_2)}\Big\}.
  \end{equation}
  We now observe that, since $X_1,\ldots,X_m$ are $\dela$-homogeneous
  of degree $1$, one has
  \begin{equation} \label{eq.siccomeXihomog}
  \begin{split}
    & \|\LL v_R\|_{W_{X}^{i,p}(B_2)}
    = \sum_{|I|\leq i} R^{2+|I|}\,\Big\| \big(X_I(\LL u)\big)\circ\delta_R\Big\|_{L^p(B_2)},\\
   &
   \|D^{i+2} v_R\|_{L^p(B_1)}
  = R^{i+2}\,\sum_{|I|=i+2} \Big\| \big(X_I u\big)\circ\delta_R\Big\|_{L^p(B_1)};
  \end{split}
  \end{equation}
  thus, by inserting \eqref{eq.siccomeXihomog} in \eqref{eq.estimThmAvR}, we obtain
  \begin{align*}
   & R^{i+2-q/p}\,\|D^{i+2} u\|_{L^p(B_R)}
   \leq c_{i,p}\,\bigg\{
   \sum_{j = 0}^iR^{2+j-q/p}\,\|D^j(\LL u)\|_{L^p(B_{2R})} +
    R^{-q/p}\,\|u\|_{L^p(B_{2R})}\bigg\}.
  \end{align*}
  Finally, since this last inequality clearly implies that
  $$\|D^{i+2} u\|_{L^p(B_R)}
  \leq c_{i,p}\,\bigg\{\sum_{j = 0}^iR^{j-i}\,\|D^j(\LL u)\|_{L^p(B_{2R})} +
    R^{-i-2}\,\|u\|_{L^p(B_{2R})}\bigg\},$$
  upon letting $R\to\infty$, we derive
  (remind that $u\in W^{k+2,p}_X(\R^n)$ and that $c_{i,p}$ is independent of $R$)
  $$\|D^{i+2} u\|_{L^p(\R^n)}\leq c_{i,p}\,\|D^i(\LL u)\|_{L^p(\R^n)}.$$
  This readily gives the desired
  \eqref{eq.DkSobolevLemma} with $\Theta_{k,p} := \max_{i = 0,\ldots,k}c_{i,p}$.
\end{proof}
 With Lemma \ref{lem.Dksobolev} at hand, we can prove the following
 global estimates for $\LL$.
 \begin{theorem}[Global $W^{k+2,p}_X$-estimates for $\LL$]\label{Thm global estimates}
 Let $p\in(1,\infty)$ and let $k$ be a nonnegative integer.
 There exists a constant $ \Lambda_{k,p} > 0$ such that, if
  $u\in W^{k+2,p}_X(\R^n)$, then
 \begin{equation} \label{global estimate}
 \|u\|_{W_{X}^{k+2,p}(\mathbb{R}^{n})} \leq
 \Lambda_{k,p}\,\Big\{\|\LL u\|_{W_{X}^{k,p}(\mathbb{R}^{n})}
  +\|u\|_{L^{p}(\mathbb{R}^{n})}\Big\}.
\end{equation}
\end{theorem}
 \begin{proof}
  By crucially exploiting Lemma \ref{lem.Dksobolev}, we have the estimate
  \begin{equation} \label{eq.globalestimateStepI}
  \begin{split}
   \|u\|_{W^{k+2,p}(\R^n)}
   & \,\,\,=\,\,\|u\|_{L^p(\R^n)}+\|Du\|_{L^p(\R^n)}
   + \sum_{i = 0}^k\|D^{i+2}u\|_{L^p(\R^n)} \\
   & \stackrel{\eqref{eq.DkSobolevLemma}}{\leq}
   \|u\|_{L^p(\R^n)}+ \|Du\|_{L^p(\R^n)}+\Theta_{k,p}\,\sum_{i = 0}^k\|D^i(\LL u)\|_{L^p(\R^n)}\\
   & \,\,\,=\,\,
   \|u\|_{L^p(\R^n)}+ \|Du\|_{L^p(\R^n)}+\Theta_{k,p}\,\|\LL u\|_{W^{k,p}_X(\R^n)}.
  \end{split}
  \end{equation}
  On the other hand, by using the global
  interpolation inequality \eqref{Prop global interpolationequata}
  (with $\varepsilon = 1$), we have
  \begin{equation} \label{eq.globalestimateStepII}
  \begin{split}
   \|Du\|_{L^p(\R^n)}
   & = \sum_{j = 1}^m\|X_ju\|_{L^p(\R^n)}
   \leq \sum_{j = 1}^m\|X_j^2u\|_{L^p(\R^n)}+
   m\,\c_p\,\|u\|_{L^p(\R^n)} \\
   & \leq \|D^2u\|_{L^p(\R^n)}
   + m\,\c_p\,\|u\|_{L^p(\R^n)}\leq \qquad \big(\text{by Lemma \ref{lem.Dksobolev} with $i = 0$}\big) \\
   & \leq \Theta_{k,p}\,\|\LL u\|_{L^p(\R^n)}+  m\,\c_p\,\|u\|_{L^p(\R^n)}.
   \end{split}
  \end{equation}
  Gathering together \eqref{eq.globalestimateStepI} and
  \eqref{eq.globalestimateStepII}, we obtain
  \eqref{global estimate} (with $\Lambda_{k,p} = \max\{2\,\Theta_{k,p},m\,\c_p+1\}$).
 \end{proof}
 We now turn to demonstrate the last ingredient for the proof
 of Theorem \ref{Ch8-Thm Sobolev Caso A}:
 \begin{theorem}[Global Sobolev regularity theorem for $\LL$] \label{Thm regularity}
  Let $p\in(1,\infty)$ and let
  $k$ be a nonnegative integer.
  Suppose that $u\in L^{p}(\mathbb{R}^{n})$ is such that $\LL u\in W_{X}^{k,p}(\mathbb{R}^{n})$
  \emph{(}meaning that the distribution $\LL u$
   can be identified with a function belonging to $W_{X}^{k,p}(\mathbb{R}^{n})$\emph{)}.

   Then $u\in W_{X}^{k+2,p}(\mathbb{R}^{n})$.
\end{theorem}
 By combining Theorems \ref{Thm global estimates} and
 \ref{Thm regularity}, we can readily provide the
 \begin{proof}[Proof of Theorem \ref{Ch8-Thm Sobolev Caso A}]
 Let $u\in L^p(\R^n)$ be such that $\LL u\in W^{k,p}_X(\R^n)$ (for some
 $p\in (1,\infty)$ and some integer $k\geq 0$). On account of Theorem \ref{Thm regularity},
 we have that
 $$u\in W^{k+2,p}_X(\R^n);$$
 as a consequence, by Theorem \ref{Thm global estimates} we have
 $$\|u\|_{W_{X}^{k+2,p}(\mathbb{R}^{n})} \leq
 \Lambda_{k,p}\,\Big\{\|\LL u\|_{W_{X}^{k,p}(\mathbb{R}^{n})}
  +\|u\|_{L^{p}(\mathbb{R}^{n})}\Big\},$$
  for a suitable constant $\Lambda_{k,p} > 0$
  independent on $u$. This ends the proof.
 \end{proof}
 We are left with the
 \begin{proof}[Proof of Theorem \ref{Thm regularity}]
  Let $u$ be as in the assertion of Theorem \ref{Thm regularity}.
  By Theorem A, $u\in W^{k+2,p}_{X,\loc}(\R^n)$; thus,
  to prove the
  theorem it suffices to show that
  \begin{equation} \label{eq.Claimtoprove}
   \| D^{i}u\|_{L^p(\R^n)}<\infty\qquad\text{for every $i=1,\ldots,k+2$}.
  \end{equation}
  To prove \eqref{eq.Claimtoprove}, we proceed by steps. \medskip

  \textsc{Step I:} We begin by proving that \eqref{eq.Claimtoprove}
  holds for $i = 2$. \medskip

  To this end, let $R > 0$ be arbitrarily fixed,
  let $\sigma\in (0,1)$ and let $\phi\in C_0^\infty(\R^n)$ be a cut-off
  function as in Lemma \ref{Lemma cutoff}, with
  $r_1 := \sigma R$ and $r_2 := \sigma'R$ (where $\sigma' := (1+\sigma)/2 < 1$).
  Since $v:=u\,\phi$ belongs to $W^{k+2,p}_X(\R^n)\subseteq W^{2,p}_X(\R^n)$,
  we can apply Theorem \ref{Thm global estimates}
  (with $k = 0$) to $v$, obtaining
  $$\|D^2(u\,\phi)\|_{L^p(\R^n)}\leq \|u\,\phi\|_{W^{2,p}_X(\R^n)}\leq \Lambda_{0,p}\,\Big\{
   \|\LL (u\,\phi)\|_{L^p(\R^n)}+\|u\,\phi\|_{L^p(\R^n)}
  \Big\}.$$
  From this, by taking into account properties (i)-to-(iii) of
  $\phi$ in Lemma \ref{Lemma cutoff}, we get
  \begin{equation*}
   \begin{split}
   & \|D^2 u\|_{L^p(B_{\sigma R})}= \|D^2 (u\,\phi)\|_{L^p(B_{\sigma R})} \leq
   \|D^2(u\,\phi)\|_{L^p(\R^n)}\leq \Lambda_{0,p}\,\Big\{
   \|\LL (u\,\phi)\|_{L^p(\R^n)}+\|u\,\phi\|_{L^p(\R^n)}
  \Big\} \\
  & \quad \leq
  \Lambda_{0,p}\,\bigg\{\bigg\|\phi\,\LL u
  + 2\sum_{j=1}^m\,X_ju\, X_j\phi
  + u\,\LL\phi\bigg\|_{L^p(B_{\sigma'R})} +
  \|u\phi\|_{L^p(B_{\sigma'R})}	\Big\} \\
  & \quad \leq\gamma_p\,\bigg\{
  \|\LL u\|_{L^p(B_{\sigma'R})}+
  \frac{1}{(1-\sigma)R}\|Du\|_{L^p(B_{\sigma'R})}
  + \bigg(\frac{1}{(1-\sigma)^2R^2}+1\bigg)\|u\|_{L^p(B_{\sigma'R})}\bigg\},
   \end{split}
  \end{equation*}
  where $\gamma_p > 0$ is a constant only depending on $p$ and on
  $\varrho_0,\varrho_1,\varrho_2$ in \eqref{eq.DkSobolevLemma}
  (hence, $\gamma_p$ is
  independent of $R$ and $\sigma$).
  We multiply both far sides of the above inequality
  by $(1-\sigma)^2R^2 > 0$,
  \begin{align*}
   &\big((1-\sigma)R\big)^2\,\|D^2u\|_{L^p(B_{\sigma R})}
   \leq\\
   &\gamma_p\,\bigg\{
  R^2\,\|\LL u\|_{L^p(B_{\sigma'R})}+
  (1-\sigma)\,R\,\|Du\|_{L^p(B_{\sigma'R})} + (1+R^2)\,\|u\|_{L^p(B_{\sigma'R})}\bigg\}.
  \end{align*}
  Due to the arbitrariness of $\sigma$,
  remembering
  the definition of $\Phi_i(u)$ (with $i = 0,1,2$) in
  \eqref{eq.defiPhik} and
  using the local interpolation
  inequality in Proposition \ref{Prop local interpolation}, we
  get
  (see also \eqref{eq.sigmaprimedausare})
  \begin{equation*}
   \begin{split}
    \Phi_2(u) & \leq \gamma_p\,\bigg\{R^2\,\|\LL u\|_{L^p(B_{R})}
    + 2\,\Phi_1(u)
    +(1+R^2)\,\|u\|_{L^p(B_{R})}\bigg\} \\
    & \big(\text{by \eqref{Prop global interpolationequata} with
    $0 < \varepsilon < \min\big\{1,(2\gamma_p)^{-1}\big\}$}\big) \\
    & \leq
    \gamma_p\,\bigg\{R^2\,\|\LL u\|_{L^p(B_{R})}
    + 2\,\varepsilon\,\Phi_2(u)
    +\Big(1+R^2+\frac{2\,\alpha_p}{\varepsilon}\Big)\|u\|_{L^p(B_{R})}\bigg\}.
   \end{split}
  \end{equation*}
  As a consequence (isolating $\sigma=1/2$ in the definition
  of $\Phi_2(u)$), we obtain
  \begin{align*}
   \|D^2 u\|_{L^p(B_{R/2})} & = \frac{4}{R^2}\bigg(
   \frac{R^2}{4}\,\|D^2 u\|_{L^p(B_{R/2})}\bigg) \leq \frac{4}{R^2}\,\Phi_2(u) \\
    & \leq \frac{4\,\gamma_p}{1-2\,\varepsilon\,\gamma_p}\,\bigg\{
   \|\LL u\|_{L^p(B_{R})}
    +\Big(\frac{1}{R^2}+1+\frac{2\,\alpha_p}{\varepsilon\,R^2}\Big)\|u\|_{L^p(B_{R})}
    \bigg\}.
  \end{align*}
  Finally, letting $R\to\infty$ (and remembering that
  $\gamma_p$ does not depend on $R$), one has
  $$\|D^2 u\|_{L^p(\R^n)}\leq
  \frac{4\,\gamma_p}{1-2\,\varepsilon\,\gamma_p}\,\bigg\{
  \|\LL u\|_{L^p(\R^n)}+\|u\|_{L^p(\R^n)}\bigg\},$$
  and this proves that
  $\| D^{2}u\|_{L^p(\R^n)}<\infty$ (since, by assumption, both $u$ and $\LL u$
  belong to $L^p(\R^n)$). \medskip

  \textsc{Step II:} We now prove that \eqref{eq.Claimtoprove}
  holds for $i = 1$.  To this end, let $R > 0$ be arbitrarily fixed.
  Since $u\in W^{k+2,p}_{X,\loc}(\R^n)$, we know that $u\in W^{2,p}_{X,\loc}(B_R)$.
  In due course
  of the proof of Proposition \ref{Prop global interpolation},
  we have proved that, if $R$ is sufficiently large,
  it holds that
  (see \eqref{interp furba} with $\varepsilon = 1$)
 $$\| X_{i}u\|_{L^{p}( B_{R/4}) }\leq  \|X_{i}^{2}u\|_{L^{p}( B_{R}) }
 +\c_p\,\|u\|_{L^{p}( B_{R}) };$$
 as a consequence, we infer that
 \begin{align*}
 \|D u\|_{L^{p}( B_{R/4}) } & =
  \sum_{j = 1}^m\| X_{i}u\|_{L^{p}( B_{R/4}) }\leq
 \sum_{j = 1}^m\|X_{i}^{2}u\|_{L^{p}( B_{R}) }
 +m\,\c_p\,\|u\|_{L^{p}( B_{R}) } \\
 & \leq \|D^2 u\|_{L^{p}( B_{R}) } + m\,\c_p\,\|u\|_{L^{p}( B_{R}) }.
 \end{align*}
  By letting $R\to\infty$
  (and remembering that $\c_p$ does not depend on $R$), we get
  $$\|D u\|_{L^{p}(\R^n)}
  \leq \|D^2 u\|_{L^{p}(\R^n)} + m\c_p\,\|u\|_{L^{p}(\R^n)},$$
  and this proves that $\|D u\|_{L^{p}(\R^n)}<\infty$,
  as $u\in L^p(\R^n)$ and, by Step I, $\|D^2 u\|_{L^{p}(\R^n)}<\infty$. \medskip

  \textsc{Step III:} In this last step we show that \eqref{eq.Claimtoprove}
  holds for every $i=3,\ldots,k+2$. \medskip

  To this end, we first perform a (finite) induction
  argument on $i\in\{0,\ldots,k\}$ to prove the existence of a constant $\kappa_i > 0$,
 only  depending on $i$ (and on $k$ and $p$), such that
  \begin{equation} \label{eq.toproveinduction}
   \|D^{i+2}u\|_{L^p(B_h)}\leq
   \kappa_i\,\Big\{\|\LL u\|_{W^{i,p}_X(B_{h+1+i})}
   + \|Du\|_{L^p(B_{h+1+i})}+\|u\|_{L^p(B_{h+1+i})}\Big\},\quad \forall\,\,h\in \mathbb{N}.
     \end{equation}
     Let us start with the case $i = 0$. For any fixed $h\in\N$, we choose
     a cut-off function $\phi_h\in C_0^\infty(\R^n)$ as in Lemma \ref{Lemma cutoff},
     with $r_1 := h$ and $r_2 := h+1$, and we define
     $v_h := u\phi_h$.
     Since we already know that $u\in W^{k+2,p}_{X,\loc}(\R^n)$, we have
  $v_h \in W^{k+2,p}_X(\R^n)$;
  as a consequence, by Lemma \ref{lem.Dksobolev} (with $i = 0$),
  $$\|D^{2}v_h\|_{L^p(\R^n)} \leq \Theta_{k,p}\,\|\LL v_h\|_{L^p(\R^n)}.$$
  From this, taking into account the properties of $\phi_h$ in Lemma \ref{Lemma cutoff}, we have
  (notice that $r_2-r_1 = 1$)
  \begin{align*}
    \|D^{2} u\|_{L^p(B_h)}
    &=\|D^{2} v_h\|_{L^p(B_h)} \leq \|D^{2}v_h\|_{L^p(\R^n)} \leq \Theta_{k,p}\,\|\LL v_h\|_{L^p(\R^n)} \\
   & \leq \Theta_{k,p}\,\Big\|u\,\LL\phi_h + 2\,\sum_{j = 1}^mX_ju\,
   X_j\phi_h + \phi_h\,\LL u\Big\|_{L^p(B_{h+1})} \\
   &  \leq \kappa_1\,\bigg\{\|\LL u\|_{L^p(B_{h+1})}
   + \|u\|_{L^p(B_{h+1})}+\|Du\|_{L^p(B_{h+1})}\bigg\},
  \end{align*}
  where $\kappa_1 > 0$ is a constant only depending
  on the bounds
  $\varrho_0,\varrho_1,\varrho_{2}$ in \eqref{eq.estimDkphi}
  (hence, $\kappa_1$ does not depend on $h$). This is precisely
  the desired \eqref{eq.toproveinduction} with $i = 0$.

  Let us now take $j\in\{0,\ldots,k-1\}$ and, assuming
  that \eqref{eq.toproveinduction} holds
  for $i=0,\ldots,j$, let us prove that \eqref{eq.toproveinduction} is fulfilled
  for $i$ replaced by $j+1$. Arguing as above, with the very same $\phi_h$, by
  applying Lemma \ref{lem.Dksobolev} to the function $v_h = u\phi_h$
  (and with $i = j+1\leq k$), we obtain
  \begin{align*}
  & \|D^{j+3}u\|_{L^p(B_h)} \leq \|D^{j+3}v_h\|_{L^p(\R^n)}
  \stackrel{\eqref{lem.Dksobolev}}{\leq}
  \Theta_{k,p}\,\|D^{j+1}(\LL v_h)\|_{L^p(\R^n)}\\
   & \quad \leq \Theta_{k,p}\,\bigg\|D^{j+1}\Big(u\,\LL\phi_h + 2\,\sum_{l = 1}^mX_lu\,
   X_l\phi_h + \phi_h\,\LL u\Big)\bigg\|_{L^p(B_{h+1})} \\
   & \quad \leq \Theta'_{k,p}\,\bigg\{\|D^{j+1}(\LL u)\|_{L^p(B_{h+1})}
   + \sum_{l = 0}^{j+2}\|D^l u\|_{L^p(B_{h+1})}\bigg\} \\
   & \quad = \Theta'_{k,p}\,\bigg\{\|D^{j+1}(\LL u)\|_{L^p(B_{h+1})}
   + \|u\|_{L^p(B_{h+1})}+\|Du\|_{L^p(B_{h+1})}
   + \sum_{i = 0}^{j}\|D^{i+2} u\|_{L^p(B_{h+1})}\bigg\} = (\star),
  \end{align*}
  where $\Theta'_{k,p} > 0$ is a suitable constant independent of $h$.
  On the other hand, since we are assuming that
  \eqref{eq.toproveinduction} holds for any
   $0\leq i\leq j $ (and for every $h\geq 1$), we have
  \begin{align*}
   \|D^{i+2}u\|_{L^p(B_{h+1})}& \leq
   \kappa_i\,\Big\{\|\LL u\|_{W^{i,p}_X(B_{h+2+i})}
   + \|Du\|_{L^p(B_{h+2+i})}+\|u\|_{L^p(B_{h+2+i})}\Big\} \\
   & \leq
   \kappa_i\,\Big\{\|\LL u\|_{W^{j,p}_X(B_{h+2+j})}
   + \|Du\|_{L^p(B_{h+2+j})}+\|u\|_{L^p(B_{h+2+j})}\Big\}.
   \end{align*}
   By using this last estimate, we obtain
   \begin{align*}
    (\star) & \leq \Theta'_{k,p}\,\bigg(1+\sum_{i = 0}^{j}\kappa_i\bigg)\cdot
    \bigg\{\|D^{j+1}(\LL u)\|_{L^p(B_{h+2+j})}
   + \|u\|_{L^p(B_{h+2+j})}+\|Du\|_{L^p(B_{h+2+j})} \\
   & \qquad \quad \qquad\qquad \qquad\quad +\|\LL u\|_{W^{j,p}_X(B_{h+2+j})}
   + \|Du\|_{L^p(B_{h+2+j})}+\|u\|_{L^p(B_{h+2+j})} \bigg\} \\
   & \leq \kappa_{j+1}\,\Big\{\|\LL u\|_{W^{j+1,p}_X(B_{h+2+j})}
   + \|Du\|_{L^p(B_{h+2+j})}+\|u\|_{L^p(B_{h+2+j})}\Big\},
   \end{align*}
   where we have introduced the constant (independent of $h$)
   $\kappa_{j+1} := 2\,\Theta'_{k,p}\,\big(1+\sum_{i = 0}^{j}\kappa_i\big)$.
   This is precisely the desired \eqref{eq.toproveinduction} with $i$ replaced by $j+1$ and we are done.

 Letting $h\to \infty$ in \eqref{eq.toproveinduction}, one gets
\begin{equation*}
   \|D^{i+2}u\|_{L^p(\R^n)}\leq
   \kappa_i\,\Big\{\|\LL u\|_{W^{i,p}_X(\R^n)}
   + \|Du\|_{L^p(\R^n)}+\|u\|_{L^p(\R^n)}\Big\},\quad \text{for $i=0,\ldots,k$}.
\end{equation*}
 Since the right-hand side is finite due to Step II (and the assumption), we infer
  $\|D^{i+2}u\|_{L^p(\R^n)}<\infty$ for $i=0,\ldots,k$, and the proof is complete.
\end{proof}


\begin{thebibliography}{99}

\bibitem{BBCCM}
 S. Biagi, A. Bonfiglioli:
 \emph{A completeness result for time-dependent vector fields and applications.}
 Commun. Contemp. Math. \textbf{17} (2015), 1--26.

\bibitem{BiBo}
 S. Biagi, A. Bonfiglioli:
 \emph{The existence of a global fundamental
 solution for homogeneous H\"{o}rmander operators via a global lifting method.}
 Proc. Lond. Math. Soc. \textbf{114} (2017), 855--889.

\bibitem{BiagiBonfBook}
 S. Biagi, A. Bonfiglioli:
 ``An Introduction to the Geometrical Analysis of Vector Fields -
 with Applications to Maximum Principles and Lie Groups'',
World Scientific Publishing, Singapore (2019).

\bibitem{BLUlibro}
 A. Bonfiglioli, E. Lanconelli, F. Uguzzoni: \newblock
 ``Stratified Lie Groups and Potential
 Theory for their sub-La\-pla\-cians'', \newblock
 Springer Monographs in Mathematics \textbf{26},
 Springer, New York, N.Y., 2007.

\bibitem {BrBrbook}M. Bramanti, L. Brandolini: H\"{o}rmander operators.
Monograph, to appear.

\bibitem {BBto}M. Bramanti, L. Brandolini: $L^{p}$ estimates for uniformly
hypoelliptic operators with discontinuous coefficients on homogeneous groups.
Rend. Sem. Mat. Univ. Politec. Torino 58 (2000), no. 4, 389--433 (2003).

\bibitem {BCLP}M. Bramanti, G. Cupini, E. Lanconelli, E. Priola: Global
$L^{p}$ estimates for degenerate Ornstein-Uhlenbeck operators.
Mathematische Zeitschrift, 266, n. 4 (2010), pp. 789-816.

\bibitem {BCLP2}M. Bramanti, G. Cupini, E. Lanconelli, E. Priola: Global
$L^{ p}$ estimates for degenerate Ornstein-Uhlenbeck operators
with variable coefficients. Mathematische Nachrichten, Vol. 286, Issue 11-12
(2013), 1087--1101.

\bibitem {Fo2}G. B. Folland: Subelliptic estimates and function spaces on
nilpotent Lie groups. Ark. Mat. 13 (1975), no. 2, 161--207.

\bibitem {Fo3}G. B. Folland: On the Rothschild-Stein lifting theorem. Comm.
Partial Differential Equations 2 (1977), no. 2, 165--191.

\bibitem {Hor2}L. H\"{o}rmander: Hypoelliptic second order differential
equations. Acta Math. 119 (1967), 147-171.

\bibitem {NSW}A. Nagel, E. M. Stein, S. Wainger: Balls and metrics defined by
vector fields. I. Basic properties. Acta Math. 155 (1985), no. 1-2, 103--147.

\bibitem {RS}L. P. Rothschild, E. M. Stein: Hypoelliptic differential
operators and nilpotent groups. Acta Math. 137 (1976), no. 3-4, 247--320.
\end{thebibliography}
\end{document}